\documentclass[12pt]{article}
	\usepackage{amsmath,fullpage,amsthm,amssymb,xcolor,graphicx}
	
	\newtheorem{theorem}{Theorem}[section]
	\newtheorem{corollary}{Corollary}[section]
	\newtheorem{lemma}{Lemma}[section]
	\newtheorem{definition}{Definition}[section]
	\newtheorem{remark}{Remark}[section]

		\DeclareMathOperator{\supp}{supp}
			\DeclareMathOperator{\Ker}{Ker}
	\DeclareMathOperator{\spec}{spec}
       
       		\DeclareMathOperator{\diag}{diag}

	\title{On the multiplicity of $ A_\alpha $-eigenvalues and the rank of complex unit gain graphs}
\author{Aniruddha Samanta \thanks{Department of Mathematics, Indian Institute of Technology Kharagpur, Kharagpur 721302, India. Email: aniruddha.sam@gmail.com}\  \and M. Rajesh Kannan\thanks{Department of Mathematics, Indian Institute of Technology Kharagpur, Kharagpur 721302, India. Email: rajeshkannan1.m@gmail.com, rajeshkannan@maths.iitkgp.ac.in }
}
\date{\today}
\begin{document}
\maketitle
\baselineskip=0.25in

\begin{abstract}

Let $ \Phi=(G, \varphi) $ be a connected complex unit gain graph ($ \mathbb{T} $-gain graph) on a simple  graph $ G $ with $ n $ vertices and maximum vertex degree $ \Delta $. The associated adjacency matrix and degree matrix are denoted by $ A(\Phi) $ and $ D(\Phi) $, respectively.  Let $ m_{\alpha}(\Phi,\lambda) $ be the multiplicity of $ \lambda $ as an eigenvalue of $ A_{\alpha}(\Phi) :=\alpha D(\Phi)+(1-\alpha)A(\Phi)$, for $ \alpha\in[0,1) $. In this article, we establish that
$	m_{\alpha}(\Phi, \lambda)\leq \frac{(\Delta-2)n+2}{\Delta-1}$, and characterize the classes of graphs for which the equality hold. Furthermore,  we establish a couple of  bounds for the rank of $A(\Phi)$ in terms of the maximum vertex degree and the number of vertices.  One of the main results extends a result known for unweighted graphs and simplifies the proof in \cite{Wang(Extnd)}, and other results provide better bounds for $r(\Phi)$ than the bounds known in \cite{gain_rank(2020)}.

%
%Let $ \Phi=(G, \varphi) $ be a connected complex unit gain graph ($ \mathbb{T} $-gain graph) on an underlying graph $ G $ with $ n $ vertices and the maximum vertex degree $ \Delta $. Then the rank of $ \Phi $, $ r(\Phi) $ is considered as the rank of its associated adjacency matrix.

%	A unit complex gain graph (or $ \mathbb{T} $-gain graph) on an undirected graph $ G $ is an ordered pair $(G,\varphi)$, denoted by $ \Phi=(G, \varphi) $ such that the function $ \varphi $ assigns a unit complex number to each orientation of an edge of $ G $ and its inverse is assigned to the opposite orientation. The associated adjacency matrix $ A(\Phi) $ is defined canonically. The rank of $ \Phi $, $ r(\Phi) $ is defined to be the rank of $ A(\Phi) $.
%In this article,
%	 we study some relationship between rank of $ \Phi$ and the zero forcing number of $ G $. For any connected $ \mathbb{T} $-gain graph $ \Phi $,

% we establish that $ r(\Phi)\geq \frac{n-2}{\Delta-1} $. Furthermore, we characterize the class of $ \mathbb{T} $-gain graphs  for which $  r(\Phi)= \frac{n-2}{\Delta-1}   $ holds. This bound improve the bound of Yong Lu, Jingwen Wu [Linear Algebra Appl. 610 (2021), 73–85; MR4159284, \cite{gain_rank(2020)} ] for connected $ \mathbb{T} $-gain graphs. Furthermore, we provide some more lower bounds of $ r(\Phi) $ in terms of $ \Delta $ and $ n $.\break
	\end{abstract}



\section{Introduction}
Let $G = (V(G), E(G))$ be a simple graph where $ V(G)=\{v_1, v_2, \dots, v_n\} $ and $ E(G) $ are the vertex set and the edge set of $ G $, respectively. If two vertices $ v_s $ and $ v_t $ are connected by an edge, we write $ v_s\sim v_t $. If $ v_s\sim v_t $, then the edge between them is denoted by $ e_{s,t} $. The \emph{degree} of a vertex $ v_s $ is denoted by $ d(v_s) $ and is defined as the number of vertices adjacent to $ v_s $. Then the \emph{maximum vertex degree} of $ G $ is denoted by  $ \Delta(G) $ (or, simply $ \Delta $). The \emph{degree matrix} of a graph $ G $ is a diagonal matrix, denoted by $ D(G) $,  is defined by $ D(G) :=\diag(d(v_1),d(v_2), \dots, d(v_n)) $. The \emph{adjacency matrix } $ A(G) $ of a graph $ G $ is a symmetric matrix whose $ (s,t)th $ entry is $ 1 $ if $ v_s\sim v_t $, and zero otherwise. The nullity of $ A(G) $ is the multiplicity of zero eigenvalue of $ A(G) $, and  is called the \emph{nullity of $ G $}, denoted by $ \eta(G) $. The \emph{rank} of $ G $ is the rank of $ A(G) $, and is denoted by $ r(G) $. Thus $ \eta(G)=n-r(G) $.

Let $ G $ be a simple undirected graph. An oriented edge from the vertex $ v_s $ to the vertex $ v_t $ is denoted by $ \overrightarrow{e_{s,t}} $. For each undirected edge $ e_{s,t}\in E(G) $, there is a pair of oriented edges $ \overrightarrow{e_{s,t}} $ and $ \overrightarrow{e_{t,s}} $. The collection $ \overrightarrow{E(G)}:=\{ \overrightarrow{e_{s,t}},\overrightarrow{e_{t,s}}: e_{s,t}\in E(G)\} $ is  the \emph{oriented edge set associated with $ G $}. Let $ \mathbb{T}=\{ z\in \mathbb{C}: |z|=1\} $.  A \emph{complex unit gain graph (or $ \mathbb{T} $-gain graph)} on a simple graph $ G $ is an ordered pair $ (G, \varphi) $, where the gain function $ \varphi: \overrightarrow{E(G)} \rightarrow \mathbb{T} $ is a mapping  such that $ \varphi( \overrightarrow{e_{s,t}}) =\varphi(\overrightarrow{e_{t,s}})^{-1}$, for every $ e_{s,t}\in E(G) $. A $ \mathbb{T} $-gain graph $ (G, \varphi) $ is  denoted by $ \Phi $. The \emph{ adjacency matrix} of  a $ \mathbb{T} $-gain graph $ \Phi=(G, \varphi)$ is  a Hermitian  matrix, denoted by $ A(\Phi)$ and its $ (s,t)th $ entry is defined as follows:
$$A(\Phi)_{st}=\begin{cases}
	\varphi(\overrightarrow{e_{s,t}})&\text{if } \mbox{$v_s\sim v_t$},\\
	0&\text{otherwise.}\end{cases}$$

The rank and the nullity of $ A(\Phi) $ are  the \emph{rank} and \emph{nullity} of $ \Phi $, denoted by $ r(\Phi) $ and $ \eta(\Phi) $, respectively. The degree matrix $ D(\Phi) $ and maximum vertex degree $ \Delta(\Phi) $ of $ \Phi $ are same as $ D(G) $ and $ \Delta(G) $, respectively. The notion of adjacency matrix of $\mathbb{T}$-gain graphs generalize the notion of adjacency matrix of undirected graphs, adjacency matrix of  signed graphs and the Hermitian adjacency matrix of a digraph.  The notion of gain graph was introduced in \cite{gain-genesis}. For more information about the properties of gain graphs and $\mathbb{T}$-gain graphs, we refer to \cite{ Our-paper-1,reff1,Reff2016,Our-paper-2, Zas4, Zaslav}.

Let $ G $ be a simple graph with adjacency matrix $A(G) $ and degree matrix $ D(G) $. In \cite{Niki}, the author introduced the following new family of matrices, denoted by $ A_{\alpha}(G) $, associated with an undirected graph:
\begin{center}
	$ A_{\alpha}(G)=\alpha D(G)+(1-\alpha)A(G),  ~~~ \alpha\in[0,1].$
\end{center}
Then $ A_0(G)=A(G) $, $ A_{\frac{1}{2}}(G)=\frac{1}{2}Q(G) $ and  $ A_1(G)=D(G) $, where $ Q(G) $ is the signless Laplacian of $ G $.  The eigenvalues of $ A_{\alpha}(G) $ are known as \emph{$ A_\alpha $-eigenvalues} of $ G $. For more details about $ A_\alpha $-eigenvalues, we refer to \cite{Niki, Wang(Extnd)}. Let $ \Phi=(G, \varphi) $ be a $ \mathbb{T} $-gain graph on underlying graph $ G $. In \cite{Shuchao_Li}, the authors introduced $ A_\alpha $-matrix for a complex unit gain graph $ \Phi $, denoted by $ A_{\alpha}(\Phi) $, defined as follows:
\begin{center}
	$ A_{\alpha}(\Phi)=\alpha D(\Phi)+(1-\alpha)A(\Phi),~~~~ \alpha\in[0,1].$
\end{center}
The eigenvalues of $ A_{\alpha}(\Phi) $ are called \emph{$ A_\alpha $-eigenvalues} of $ \Phi $. Note that $ A_{\alpha}(\Phi) $ is a generalization of $ A_{\alpha}(G) $ and hence the matrices $ A_{\alpha} (G)$, $ A(\Phi) $, $ A(G) $ are particular cases of $ A_{\alpha}(\Phi) $. Let $ \lambda $ be an eigenvalue of $ A_{\alpha}(\Phi) $. Then the \emph{multiplicity of $ \lambda $} is denoted by $ m_{\alpha}(\Phi, \lambda) $.

In \cite{Wang(Extnd)}, authors proved the following upper bound for the multiplicity of $ A_\alpha $-eigenvalues of a connected graph $ G $.

\begin{theorem}\cite[Theorem 3.1, Theorem 3.3]{Wang(Extnd)} \label{Th1}
	Let $ G $ be a connected graph of $ n $ vertices with maximum vertex degree $ \Delta\geq 2 $. If $ m_{\alpha}(G, \lambda) $ is the multiplicity of $ \lambda $ as an eigenvalue of $ A_{\alpha}(G)$, then for $ \alpha \in [0,1) $,
	\begin{equation*}
		m_{\alpha}(G, \lambda)\leq\frac{(\Delta-2)n+2}{\Delta-1}
	\end{equation*}
	Equality occur if and only if  $ G $ and $ \lambda $ satisfy one of the following:
	\begin{enumerate}
		\item[(i)] $ G=K_n $ and $ \lambda=\alpha n-1 $,
		\item [(ii)] $ G=C_n $ with even order and $ \lambda \in \{2\alpha+2(1-\alpha)\cos(\frac{2\pi j}{n}): j=1,2, \dots, \frac{n-2}{2} \} $,
		\item [(iii)] $ G=C_n $ with order order and $ \lambda \in \{2\alpha+2(1-\alpha)\cos(\frac{2\pi j}{n}): j=1,2, \dots, \frac{n-1}{2} \} $,
		\item[(iv)] $ G=K_{\frac{n}{2}, \frac{n}{2}} $ and $ \lambda=\frac{\alpha n}{2} $.
	\end{enumerate}
\end{theorem}

One of the main objectives of this article is to extend Theorem \ref{Th1} for $ A_\alpha $-matrices of $ \mathbb{T} $-gain graphs, and provide an alternate simple proof  (Theorem \ref{main_result}).

Establishing bounds for the rank and the nullity of a graph, in terms of $ \Delta $ and $ n $, is an interesting problem considered in literature \cite{Nullity2019(LAA), Nullity2020(JGT), Nullity2019(LAA2), Nullity2018(LAA)}.	The following bounds for the rank of $\mathbb{T}$-gain graphs are known.
\begin{theorem}[{\cite[Theorem 3.2]{gain_rank(2020)}}]\label{Th1.1}
	Let $ \Phi=(G, \varphi) $ be any $ \mathbb{T} $-gain graph of $ n $ vertices with maximum vertex degree $ \Delta $. Then
	\begin{equation*}
		r(\Phi)\geq \frac{n}{\Delta}.
	\end{equation*}
\end{theorem}

%{ \bf I ignore the theorem of equality characterization of the above result because the above lower bound coincide with our lower bound when $ n=2\Delta $ and this is the case when equality occurs.}
%
%\begin{theorem}[{\cite[Theorem 3.3]{gain_rank(2020)}}]\label{Th1.2}	Let $ \Phi=(G, \varphi) $ be any $ \mathbb{T} $-gain graph with $ n $ vertices and maximum vertex degree $ \Delta $. Then $ r(\Phi)=\frac{n}{\Delta} $ if and only if $ \Phi=\frac{n}{2\Delta}K^{\varphi}_{\Delta, \Delta} $, and each $ C^{\varphi}_4 $(if any) in $ K^{\varphi}_{\Delta, \Delta} $ is of Type A.
%\end{theorem}

%
%If the $ \mathbb{T} $-gain graph $ \Phi $ is connected, then $ r(\Phi)\geq 2 $. Then Theorem \ref{Th1.1} gives a better bound than the trivial bound $ 2 $ if and only if $ \frac{n}{\Delta}\geq 2 $. Therefore, it is to be assumed that $ n\geq 2\Delta $. In this article, we establish that for any connected $ \mathbb{T} $-gain graph $ \Phi $,  $ r(\Phi)\geq \frac{n-2}{\Delta-1} $ and characterize the $ \mathbb{T} $-gain graphs for which equality occur [Theorem \ref{main-thm1}]. Since for $ n\geq 2\Delta $, $ r(\Phi) \geq \frac{n-2}{\Delta-1}\geq \frac{n}{\Delta}$. Therefore, Theorem \ref{main-thm1} provides an improvement lower bound of $ r(\Phi) $ for connected graph.

\begin{theorem}[{\cite[Theorem 3.4]{gain_rank(2020)}}]\label{Th1.3}	
	Let $ \Phi=(G, \varphi) $ be a $ \mathbb{T} $-gain graph with $ n $ vertices and maximum vertex degree $ \Delta $ such that $ 2\Delta \nmid n $. Then
	\begin{equation*}
		r(\Phi)\geq \frac{n+1}{\Delta}.
	\end{equation*}
	Equality occur if and only if $ \Phi=\frac{n-2\Delta+1}{2\Delta}K^{\varphi}_{\Delta, \Delta} \cup K^{\varphi}_{(\Delta-1), \Delta}$ and each $ C^{\varphi}_{4} $(if any) in $ K^{\varphi}_{\Delta, \Delta}  $ and $K^{\varphi}_{(\Delta-1), \Delta}  $ is of Type A.
\end{theorem}

%\begin{theorem}
%\textbf{	Include theorem 3. 4 of \cite{gain_rank(2020)}, and include the corresponding remarks in the next paragraph.}
%\end{theorem}
The second objective of this article is to improve these bounds for $\mathbb{T}$-gain graphs (Theorem \ref{main-thm1} and Theorem \ref{th3.2}). The proofs of the main results use some known results from the field of zero-forcing sets, and in \cite{sun-li} the authors used the same techniques to derive bounds for the nullity of the adjacency matrices of graphs.

%==============================================	
\section{Definitions, notation and preliminary results}\label{prelim}
The notion of a zero-forcing set of a simple graph $ G $ is introduced in \cite{Zero-forcing}.
\begin{definition}[{\cite[Definition 2.1]{Zero-forcing}}]
	[Color-change rule] Let $ G $ be a simple graph such that each vertex of $ G $ is colored either black or white. Suppose vertex $ v $ is a black vertex and exactly one neighbor $ w $ of $ v $ is white among all other neighbors. Then change the color of $ w $ to black.
\end{definition}
The \emph{derived coloring} of a given coloring of $G$ is the resulting coloring after applying the color-change rule such that no more changes are possible. A subset $ Z $ of the vertex set of $ G $ is called a \emph{zero forcing set} of $ G $, if initially the vertices of $ Z $ are all colored black and the remaining vertices are colored white, the derived coloring of $ G $ are all black. The zero forcing number of $G$ is defined as $ Z(G)=\min\limits_{\text{ Z}}|Z|$, $ Z $ is a zero-forcing set of $ G $.

The cycle and the complete graph on $ n $ vertices are denoted by $ C_n $ and $ K_n $, respectively. A complete bipartite graph, of partition size $ m$ and $n $, is denoted by $ K_{m,n} $. In \cite{Gentner_et_al(2016)} and \cite{Genter_Micael_Dieter(2018)}, the authors established upper bounds for the zero forcing number of a graph in terms of the maximum vertex degree.
\begin{theorem}[{\cite[Theorem 1(ii)]{Gentner_et_al(2016)}}]\label{lm2.1}
	If $ G $ is a connected graph of $ n $ vertices with maximum vertex degree $ \Delta \geq 2 $. Then
	\begin{equation*}
		Z(G)\leq \frac{(\Delta-2)n+2}{\Delta-1}.
	\end{equation*}
	Equality occurs if and only if $ G $ is either $ C_n $, or $ K_n $, or $ K_{\frac{n}{2}, \frac{n}{2}} $.
	
\end{theorem}

\begin{figure} [!htb]
	\begin{center}
		\includegraphics[scale= 0.60]{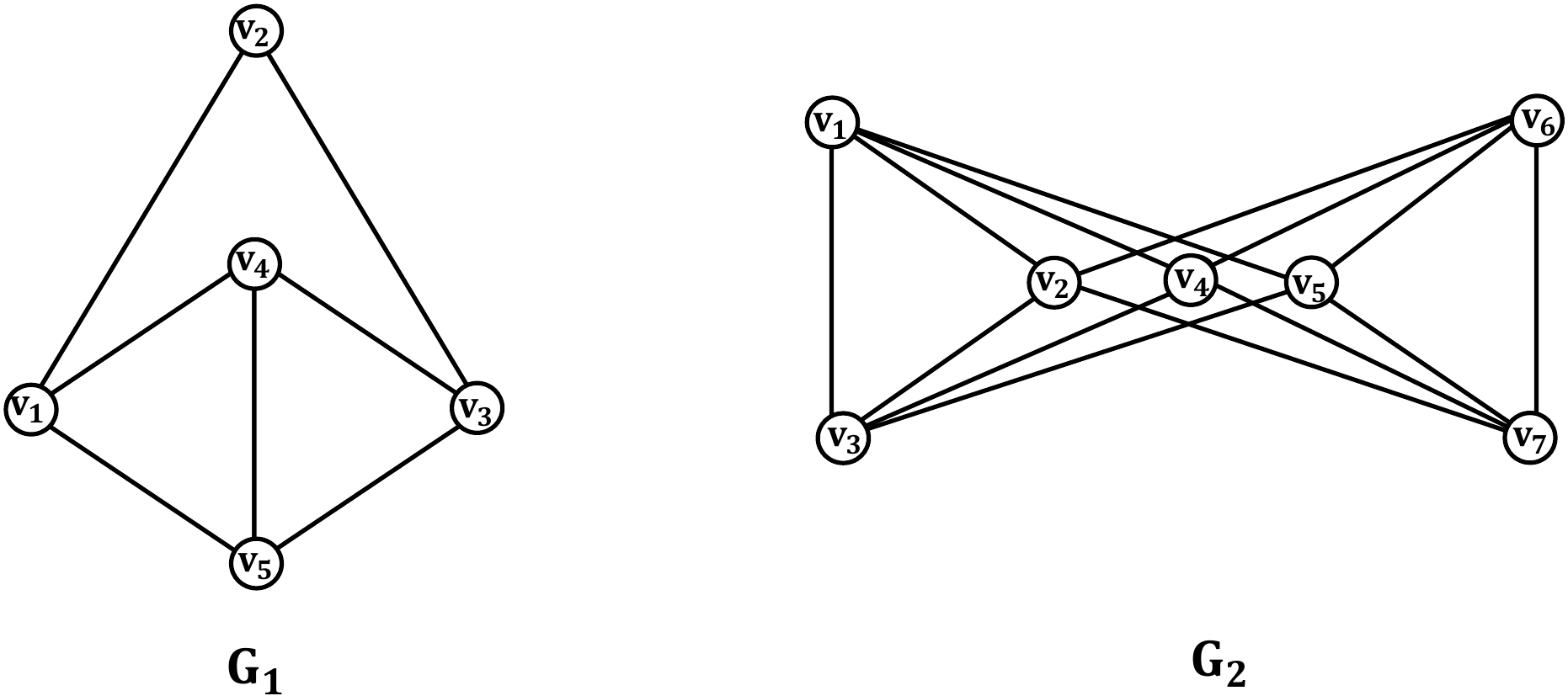}
		\caption{ Graph $ G_1 $ and $ G_2 $} \label{fig1.1}
	\end{center}
\end{figure}
\begin{theorem}[{\cite[Theorem 4]{Genter_Micael_Dieter(2018)}}]\label{th2.1}
	Let $ G $ be any connected graph of $n$ vertices and the maximum vertex degree $\Delta\geq 3$. Then
	\begin{equation*}
		Z(G)\leq \frac{(\Delta-2)n}{\Delta-1}
	\end{equation*}
	holds if and only if $G  \notin \{ G_1, G_2, K_n, K_{\frac{n}{2},\frac{n}{2}}, K_{\frac{n+1}{2},\frac{n-1}{2}}\}$, where $ G_1 $ and $ G_2 $ are given in Figure \ref{fig1.1}.
\end{theorem}

\begin{lemma}[{\cite[Proposition 2.2]{Zero-forcing}}]\label{lm2.2}
	Let $B$ be any square matrix on some field with $\eta(B)>s$. Then there exists a nonzero vector $y\in \Ker(B)$ vanishing at $s$ specified positions.
\end{lemma}

For a graph $G$, any function $ \zeta: V(G) \rightarrow \mathbb{T} $ is  called a \emph{switching function}. Let $ \Phi_1=(G, \varphi_1) $ and $ \Phi_2=(G, \varphi_2) $ be two $ \mathbb{T} $-gain graphs. Then $ \Phi_1 $ and $ \Phi_2 $ are  \emph{switching equivalent}, denoted by $ \Phi_1 \sim \Phi_2 $, if there exists a switching function $ \zeta $ such that $ \varphi_1(\overrightarrow{e_{i,j}})= \zeta(v_i)^{-1} \varphi_2(\overrightarrow{e_{i,j}})\zeta(v_j)$, for all $ e_{i,j}\in E(G) $. If $ \Phi_1 \sim \Phi_2 $, then $ A(\Phi_1) $ and $ A(\Phi_2) $ are diagonally similar and hence have the same spectra. Let $ \overrightarrow{C_n}$ be an oriented cycle in a $ \mathbb{T} $-gain graph $ \Phi=(G, \varphi) $ with oriented edges $ \overrightarrow{e_1}, \overrightarrow{ e_2}, \cdots, \overrightarrow{e_n} $, then $ \varphi(\overrightarrow{C_n})=\prod\limits_{i=1}^{n}\varphi(\overrightarrow{e_i}) $. If $ \varphi(\overrightarrow{C_n}) $ is a real number, then we simply write $ \varphi(C_n) $. A $ \mathbb{T} $-gain graph $ \Phi=(G, \varphi) $ is  \emph{balanced} if $ \varphi(C)=1 $ for all cycle $ C $ in $ \Phi $. If $ \Phi $ is balanced then we write $ \Phi\sim(G,1) $. It is known that $ (C_n, \varphi_1) \sim (C_n, \varphi_2)$ if and only if $ \varphi_1(\overrightarrow{C_n})=\varphi_2(\overrightarrow{C_n}) $. Therefore, if $ \varphi_1(\overrightarrow{C_n})=\varphi_2(\overrightarrow{C_n}) $, then $ (C_n, \varphi_1)$ and  $(C_n, \varphi_2)$  have the same spectra.

\begin{definition}[{\cite[Definition 2]{Lu_Wang_Xiao_Peng}}]
	Let $ \Phi=(C_n, \varphi) $ be any $ \mathbb{T} $-gain graph on a cycle $ C_n $. Then $ \Phi $ is called
	
	\vspace{5pt}
	\begin{center}
		$ \left
		\{
		\begin{array}{cl}
			\mbox{Type A,} & \mbox{if $ n $ is even and $ \varphi(C_n)=(-1)^{\frac{n}{2}} $}\\
			\mbox{Type B,} & \mbox{if $ n $ is even and $ \varphi(C_n)\ne (-1)^{\frac{n}{2}}$}\\
			\mbox{Type C,} & \mbox{if $ n $ is odd and $ Re((-1)^{\frac{(n-1)}{2}}\varphi(C_n))>0$}\\
			\mbox{Type D,} & \mbox{if $ n $ is odd and $ Re((-1)^{\frac{(n-1)}{2}}\varphi(C_n))<0$}\\
			\mbox{Type E,} & \mbox{if $ n $ is odd and $ Re((-1)^{\frac{(n-1)}{2}}\varphi(C_n))=0.$}\\
		\end{array}\right. $
	\end{center}
\end{definition}
The rank of the gain adjacency matrices of cycles are known.
\begin{theorem}[{\cite[Theorem 7]{Yu_Qu_et_al}}]\label{lm2.3}
	Let $ \Phi=(C_n, \varphi) $ be any $ \mathbb{T} $-gain graph on $ C_n $ with $ n $ vertices. Then
	\begin{center}
		$ r(\Phi)= \left \{ \begin{array}{cl}
			n-2, & \mbox{if $ \Phi $ is Type A}\\
			n, & \mbox{if $ \Phi $ is Type B}\\
			n, & \mbox{if $ \Phi $ is Type C}\\
			n, & \mbox{if $ \Phi $ is Type D}\\
			n-1, & \mbox{if $ \Phi $ is Type E.}
		\end{array} \right.$
	\end{center}
\end{theorem}

The following results will be used in the proofs of the main theorems.
\begin{lemma}[{\cite[Lemma 3(iii)]{Yu_Qu_et_al}}]\label{lm2.4}
	Let $ \Phi=(G, \varphi) $ be a $ \mathbb{T} $-gain graph and $ \Phi_1 $ be an induced subgraph of $ \Phi $. Then $r(\Phi_1)\leq r(\Phi) $.
\end{lemma}

\begin{theorem}[{\cite[Theorem 4.1]{Our-paper-1}\label{Th2.4}}]
	Let $ \Phi=(G, \varphi) $ be a $ \mathbb{T}$-gain graph on a bipartite graph $ G $. Then the eigenvalues of $ A(\Phi) $ are symmetric about origin.
\end{theorem}	

Let us recall a couple of results about the $ A_{\alpha} $-eigenvalues of a $ \mathbb{T} $-gain graph $ \Phi $.
\begin{lemma}[{\cite[Lemma 2.9]{Shuchao_Li}}]
	Let $ \Phi=(C_n, \varphi) $ be a $ \mathbb{T} $-gain graph such that $ \varphi(\overrightarrow{C_n})=e^{i\theta}$. Then the $ A_\alpha $-eigenvalues of $ \Phi $ are
	\begin{equation*}
		\left\{ 2\alpha+2(1-\alpha)\cos\left( \frac{\theta+2\pi j}{n}\right): j=0,1, \dots, (n-1)\right\}.
	\end{equation*}
\end{lemma}
\begin{corollary} [{\cite[Corollary 2.12]{Shuchao_Li}\label{Cor4.1}}]
	Let $ \Phi=(C_n, \varphi) $ be a $ \mathbb{T} $-gain graph such that $ \varphi(\overrightarrow{C_n})=e^{i\theta} $.  If $ m_{\alpha}(\Phi, \lambda) $ is the multiplicity of $ \lambda $ as an eigenvalue of $ A_{\alpha}(\Phi) $, then $ m_{\alpha}(\Phi,\lambda) \leq 2$, for $ \lambda \in \mathbb{R} $ and $ \alpha \in [0,1) $. Equality occur if and only if any one of the following holds:
	\begin{enumerate}
		\item[(i)] $ \theta =0 $ and $ \lambda \in \left\{ 2\alpha+2(1-\alpha)\cos\left( \frac{2\pi j}{n}\right): j=0,1, \dots, \lceil\frac{n}{2}\rceil-1 \right\}, $
		\item [(ii)] $ \theta =\pi $ and $ \lambda \in \left\{ 2\alpha+2(1-\alpha)\cos\left( \frac{(2j+1)\pi}{n}\right): j=0,1, \dots, \lfloor\frac{n}{2}\rfloor-1 \right\}. $
	\end{enumerate}
\end{corollary}
%=====================================================================
\section{Multiplicity of an $ A_\alpha $-eigenvalue of $ \mathbb{T} $-gain graph}

Let $H_n$  denote the set of all Hermitian matrices of order $n$.   For $B \in H_n$,   $\mathcal{G}(B)$ is the matrix with $ (i,j)th $ entry defined as follows: $$\mathcal{G}(B)_{ij}=\begin{cases}
	\frac{B_{ij}}{|B_{ij}|} &\text{if } \mbox{$B_{ij}\ne 0$},\\
	0&\text{otherwise.}\end{cases}$$
where $B_{ij}$ denote the $(i,j)$th entry of the matrix $B$. Let $\Phi=(G, \varphi)$ be any $\mathbb{T}$-gain graph of $n$ vertices.  A matrix $B=(B_{ij})\in H_n$ is  a \textit{matrix of type $\Phi$} if   $\mathcal{G}(B)_{ij} = A(\Phi)_{ij}$ for all $i \neq j$. Define $\mathcal{H}(\Phi)=\{B\in H_n: B ~\mbox{is of type}~ \Phi\}$. Let $ \eta(B) $ be the nullity of the matrix $ B $. Define $M(\Phi):=\max\{\eta(B): B\in \mathcal{H}(\Phi)\}$. Then $\eta(A(\Phi))\leq M(\Phi)$. For any $\mathbb{T}$-gain graph $\Phi$, the underlying graph is denoted by $\Gamma(\Phi)$. For $ y=(y_1, y_2, \dots, y_n) \in  \mathbb{C}^{n} $, the \emph{ support} of $ y $ is the set of indices $ j $ such that $ y_j \ne 0 $, and is denoted by $ \supp(y) $. A zero forcing set of a $\mathbb{T}$-gain graph $\Phi$ is the zero forcing set of its underlying graph $ \Gamma(\Phi) $. The following lemma is  an extension of \cite[Proposition 2.3]{Zero-forcing} for the complex matrices. For the sake of completeness we include a proof here.
\begin{lemma}\label{lm3.1}
	Let $\Phi$ be any $\mathbb{T}$-gain graph, and $Z$ be a zero forcing set of $\Phi$. Let $B\in \mathcal{H}(\Phi)$ and $y\in \Ker B$ with $\supp(y)\cap Z = \phi$. Then $y=0$.
\end{lemma}
\begin{proof}
	Let $ V(\Phi) $ be the vertex set of $ \Phi $. If $ Z =V(\Phi)$, then  $ y=0 $. Suppose $ Z\subset V(\Phi) $. Since $ Z$ is a zero forcing set, so all the white vertices in $ V(\Phi)\setminus Z $ can be colored black by color change rule. Let $ v_i\in Z$ be such that it has exactly one white neighbor vertex  $ v_t $. Then the $ i $-th entry $ (By)_{i}=B_{ii}y_{i}+\sum\limits_{v_i\sim v_j}B_{ij}y_{j}=B_{it}y_{t}=0 $. Thus $ y_t=0 $. As $Z$ is a zero forcing set, so all the components of $ y $ associated with white vertices are zero. Hence $ y=0 $.
\end{proof}

The following lemma is  an extension of \cite[Proposition 2.4]{Zero-forcing} for the complex matrices. For the sake of completeness we include a proof here.
\begin{lemma}\label{lm3.2}
	Let $\Phi$ be any $ \mathbb{T}$-gain graph. Then $M(\Phi)\leq Z(\Gamma(\Phi))$.
\end{lemma}
\begin{proof}
	Let $ Z $ be a zero forcing set of $ \Phi $. Suppose that $ M(\Phi)>|Z|. $ Then there exists a matrix $ B \in \mathcal{H}(\Phi) $ such that $ \eta(B)>|Z| $. Therefore, by Lemma \ref{lm2.2}, there exist a nonzero $ y\in Ker(B) $ such that $ \supp(y)\cap Z=\phi $. By Lemma \ref{lm3.1}, we get $ y=0 $, a contradiction. Thus $ M(\Phi)\leq |Z| $, and hence $ M(\Phi)\leq Z(\Gamma(\Phi)) $.
\end{proof}

\begin{lemma}\label{lm4.2}
	Let $ \Phi=(K_n, \varphi) $ be a $ \mathbb{T} $-gain graph. If $ \mu $ is an eigenvalue of $ \Phi $ with multiplicity $ (n-1) $ if and only if $\left\{ \alpha(n-1)+(1-\alpha)\mu \right\}$ is an $ A_\alpha $-eigenvalue of $ \Phi $ with multiplicity $ (n-1) $, for $ \alpha \in [0,1) $.
\end{lemma}

%Let $ \lambda $ be an eigenvalue of $ A_{\alpha}(\Phi) $. Then $ \left( A_{\alpha}(\Phi)-\lambda I\right) $ is Hermitian, where $ I $ is the identity matrix of order $ n $. Let $ B= \left( A_{\alpha}(\Phi)-\lambda I\right) $. Then the multiplicity of $ \lambda $ is same as the nullity of $ B $. That is, $ \eta(B)=m_{\alpha}(\Phi, \lambda) $. Since $ \frac{B_{ij}}{|B_{ij}|}=\frac{(1-\alpha)A(\Phi)_{ij}}{|(1-\alpha)||A(\Phi)_{ij}|}=A(\Phi)_{ij} $, for $ i \ne j $ and $ \alpha \in [0, 1) $, so $ B \in \mathcal{H}(\Phi) $. Therefore, $ \eta(B) \leq M(\Phi) $.

Now we are ready to establish one of our main results.
\begin{theorem}\label{main_result}
	Let $ \Phi=(G, \varphi) $ be a connected $ \mathbb{T} $-gain graph of $ n $ vertices with maximum vertex degree $ \Delta\geq 2 $. If $ m_\alpha(\Phi, \lambda) $ is the multiplicity of  $ \lambda $ as an $ A_\alpha $-eigenvalue of $ \Phi $, where $ \alpha \in [0,1) $, then
	\begin{equation*}
		m_{\alpha}(\Phi, \lambda) \leq \frac{(\Delta-2)n+2}{(\Delta-1)} .
	\end{equation*}
	Equality occurs if and only if one of the following holds:
	\begin{enumerate}
		\item[(i)] $ \Phi=(K_n, \varphi) $ with $ \mu\in \spec(\Phi)$ has multiplicity $ (n-1) $ and $ \lambda= \alpha(n-1)+(1-\alpha)\mu$.
		\item[(ii)] $\Phi=(C_n, \varphi)$ with $ \varphi(C_n)=1 $ and $ \lambda \in \left\{ 2\alpha+2(1-\alpha)\cos\left( \frac{2\pi j}{n}\right): j=0,1, \dots, \lceil\frac{n}{2}\rceil-1 \right\} $.
		\item [(iii)] $\Phi=(C_n, \varphi)$ with $ \varphi(C_n)=-1 $ and $ \lambda \in \left\{ 2\alpha+2(1-\alpha)\cos\left( \frac{(2j+1)\pi}{n}\right): j=0,1, \dots, \lfloor\frac{n}{2}\rfloor-1 \right\} $.
		\item[(iv)] $ \Phi\sim(K_{\frac{n}{2}, \frac{n}{2}},1) $ and $ \lambda=\frac{\alpha n}{2} $.
	\end{enumerate}
\end{theorem}
\begin{proof}
	For $ \alpha\in [0,1) $, we have $ A_{\alpha}(\Phi)=\alpha D(\Phi)+(1-\alpha)A(\Phi) $, where $ D(\Phi) $ is the degree matrix of $ \Phi $. Let $ \lambda $ be an eigenvalue of $ A_{\alpha}(\Phi) $. Then the matrix $ B:= \left( A_{\alpha}(\Phi)-\lambda I\right) $ is Hermitian, where $ I $ is the identity matrix of order $ n $, and hence the multiplicity of the eigenvalue $ \lambda $ is same as the nullity of $ B $. That is, $ \eta(B)=m_{\alpha}(\Phi, \lambda) $. Let $ B_{ij} $ be the $ (i,j)th$-entry of the matrix $ B $. If $ B_{ij}\ne 0$, then $ \frac{B_{ij}}{|B_{ij}|}=\frac{(1-\alpha)A(\Phi)_{ij}}{|(1-\alpha)||A(\Phi)_{ij}|}=A(\Phi)_{ij} $, for $ i \ne j $ and $ \alpha \in [0, 1) $. Thus $ B \in \mathcal{H}(\Phi) $. Therefore, $ \eta(B) \leq M(\Phi) $. Now, by combining Lemma \ref{lm3.2}, Theorem \ref{lm2.1} and the fact that $ \eta(B) \leq M(\Phi) $, we have
	\begin{equation*}
		m_{\alpha}(\Phi, \lambda)=\eta(B)\leq M(\Phi)\leq Z(G)\leq \frac{(\Delta-2)n+2}{\Delta-1}.
	\end{equation*}
	If $ m_{\alpha}(\Phi, \lambda)= \frac{(\Delta-2)n+2}{\Delta-1} $, then $ Z(G)=\frac{(\Delta-2)n+2}{\Delta-1}$. Therefore, by Theorem \ref{lm2.1}, $G$ is either $K_n$ or $C_n$ or   $K_{\frac{n}{2}, \frac{n}{2}}$ .
	
	\noindent{\bf Case 1:} Suppose $ \Phi=(K_n, \varphi) $ and $ m_{\alpha}(\Phi, \lambda)= \frac{(\Delta-2)n+2}{\Delta-1}$. Then $m_{\alpha}(\Phi, \lambda) =n-1.$ Therefore, $ \lambda $ is an eigenvalue of $ A_{\alpha}(\Phi) $ with multiplicity $ (n-1) $. By Lemma \ref{lm4.2}, statement $ (i) $ holds.
	
	\noindent {\bf Case 2:} Suppose $ \Phi=(C_n, \varphi)$ and $ m_{\alpha}(\Phi, \lambda)= \frac{(\Delta-2)n+2}{\Delta-1}.$ Then $ m_{\alpha}(\Phi, \lambda)=2.$ Therefore, by Corollary \ref{Cor4.1}, either  statement $ (ii) $ or  statement $ (iii) $ holds.
	
	\noindent {\bf Case 3:} Suppose $ \Phi=(K_{\frac{n}{2}, \frac{n}{2}}, \varphi) $ and $m_{\alpha}(\Phi, \lambda)= \frac{(\Delta-2)n+2}{\Delta-1}$. Then $ m_{\alpha}(\Phi, \lambda) = n-2$. Then there is an eigenvalue $ \mu$ of $ A(\Phi) $ with multiplicity $ (n-2) $ such that $ \lambda=\frac{\alpha n}{2}+(1-\alpha)\mu $. Since $ \Phi $ is bipartite, so by Theorem \ref{Th2.4}, the eigenvalues are symmetric about origin. Then $ \mu=0 $. Therefore $ r(\Phi)=2 $. Let $(C_4, \varphi)$ be an induced $4$-cycle in $\Phi$. Using Lemma \ref{lm2.4}, $2\leq r(C_4, \varphi)\leq r(\Phi)=2$. Since $ r(C_4, \varphi)=2$, so by Lemma \ref{lm2.3}, $ (C_4, \varphi) $ is of type A and hence $ \varphi(C_4)=1 $. Therefore, any $4$-cycle in $\Phi$ is neutral. Let us take an arbitrary cycle $C_{2k}\equiv v_1-v_2-\cdots -v_{2k}$. Then \begin{align*}
		\varphi(\overrightarrow{C_{2k}})&=\varphi(\overrightarrow{e_{1,2}})\varphi(\overrightarrow{e_{2,3}})\cdots \varphi(\overrightarrow{e_{(2k-1),2k}})\\
		&=\{\varphi(\overrightarrow{e_{1,2}})\varphi(\overrightarrow{e_{2,3}})\varphi(\overrightarrow{e_{3,4}})\varphi(\overrightarrow{e_{4,1}})\}\\
		&~~~~\{\varphi(\overrightarrow{e_{1,4}})\varphi(\overrightarrow{e_{4,5}})\varphi(\overrightarrow{e_{5,6}})\varphi(\overrightarrow{e_{6,1}})\}\\
		&~~~~~~~~~~~~~~~~~~\vdots\\
		&~~~~\{\varphi(\overrightarrow{e_{1,(2k-2)}})\varphi(\overrightarrow{e_{(2k-2),(2k-1)}})\varphi(\overrightarrow{e_{(2k-1),2k}})\varphi(\overrightarrow{e_{2k,1}})\}\\
		&=1.
	\end{align*}
	Therefore $\Phi\sim(K_{\frac{n}{2}, \frac{n}{2}},1)$ and $ \lambda= \frac{\alpha n}{2}$. Thus statement $ (iv) $ holds.

	The converse is easy to verify.
\end{proof}

\begin{remark}
	It is easy to see that Theorem \ref{main_result} is an extension of  Theorem \ref{Th1}. Also, the above proof simplifies the proof of Theorem  \ref{Th1}.
\end{remark}
%=====================================================================
\section{Lower bounds of rank for connected $\mathbb{T} $-gain graph}
In this section, we establish two lower bounds for the rank of a $ \mathbb{T} $-gain graph $ \Phi $ in terms of the number of vertices $ n $ and the maximum vertex degree $ \Delta $. The first bound is a consequence of Theorem \ref{main_result}.

\begin{theorem}\label{main-thm1}
	Let $ \Phi=(G, \varphi)$ be a connected $\mathbb{T}$-gain graph of $n$ vertices with rank $r(\Phi)$ and the maximum vertex degree $\Delta\geq 2$. Then
	\begin{equation*}
		r(\Phi)\geq\frac{n-2}{\Delta-1}.
	\end{equation*}
	Equality holds if and only if $\Phi$ is either  $(K_{\frac{n}{2}, \frac{n}{2}}, 1)$ or $ (C_n,\varphi) $, where $ n $ is even and $ \varphi(C_n)=\pm 1 $.
\end{theorem}
\begin{proof}
	Let $\Phi=(G, \varphi)$ be a connected $\mathbb{T}$-gain graph. Then for $\alpha \in [0,1)$, we have\break $ A_{\alpha}(\Phi)=\alpha D(\Phi)+(1-\alpha)A(\Phi) $ and $m_{\alpha}(\Phi, \lambda)$ is the multiplicity of $ \lambda$ as an eigenvalue of $ A_{\alpha}(\Phi) $. Consider $ \alpha=0 $ and $ \lambda=0 $. Then $ A_{0}(\Phi)=A(\Phi) $ and $ m_{0}(\Phi, 0)$ is the nullity of $\Phi$. That is, $ m_{0}(\Phi, 0)=\eta(\Phi) $. Therefore, by Theorem \ref{main_result}, $ \eta(\Phi)\leq \frac{(\Delta-2)n+2}{(\Delta-1)}. $ Also $ \eta(\Phi)=n-r(\Phi) $, so
	\begin{equation}\label{eq3}
		r(\Phi)\geq \frac{n-2}{\Delta-1}.
	\end{equation}
	Since rank of the adjacency matrix of $ \Phi=(K_n, \varphi) $ is at least  $ 2 $, so by Theorem \ref{main_result}, equality occurs in \eqref{eq3} if and only if $\Phi$ is either  $(K_{\frac{n}{2}, \frac{n}{2}}, 1)$ or $ (C_n,\varphi) $, where $ n $ is even and $ \varphi(C_n)=\pm 1 $.
\end{proof}

Let $ \Phi=(G, \varphi) $ be a connected $ \mathbb{T} $-gain graph with $ n $ vertices and  the maximum vertex degree $ \Delta $. Then either $ n< 2\Delta $ or $ n \geq 2\Delta $. If $ n<2\Delta $, then $ r(\Phi)\geq2> \frac{n}{\Delta} \geq \frac{n-2}{\Delta-1} $.
%In this case, $ 2 $ is the trivial better bound than both of the bounds $ \frac{n}{\Delta}$ and $ \frac{n-2}{\Delta-1} $.
If $ n \geq 2\Delta $, then $ r(\Phi)\geq \frac{n-2}{\Delta-1}\geq \frac{n}{\Delta}$. Therefore, $ \frac{n-2}{\Delta-1} $ is better than $ \frac{n}{\Delta} $. 	Thus the bound derived in the above theorem improves the lower bound of $ r(\Phi) $ given in  Theorem \ref{Th1.1}.

%If the $ \mathbb{T} $-gain graph $ \Phi $ is connected, then $ r(\Phi)\geq 2 $. Then Theorem \ref{Th1.1} gives a better bound than the trivial bound $ 2 $ if and only if $ \frac{n}{\Delta}\geq 2 $. Therefore, it is to be assumed that $ n\geq 2\Delta $. In this article, we establish that for any connected $ \mathbb{T} $-gain graph $ \Phi $,  $ r(\Phi)\geq \frac{n-2}{\Delta-1} $ and characterize the $ \mathbb{T} $-gain graphs for which equality occur [Theorem \ref{main-thm1}]. Since for $ n\geq 2\Delta $, $ r(\Phi) \geq \frac{n-2}{\Delta-1}\geq \frac{n}{\Delta}$. Therefore, Theorem \ref{main-thm1} provides an improvement lower bound of $ r(\Phi) $ for connected graph.

%\begin{remark}
%	\textbf{It is easy to see that $\frac{n}{\Delta} \leq \frac{n-2}{\Delta-1}$ if and only if $\Delta \leq \frac{n}{2}.$ The assumptions of  Theorem \ref{main-thm1} implies that the rank of $A(\Phi)$ is greater than or equal to $2$. By combining these observations, we can conclude that whenever $\frac{n}{\Delta} \geq 2$, our bound for $r(\Phi)$ is better than that of  \cite[Theorem 3.2]{gain_rank(2020)}.}
%\end{remark}

Now we establish a bound for $ r(\Phi) $ in terms of $ n $ and $ \Delta $.
\begin{theorem}\label{th3.2}
	Let $\Phi$ be any connected $\mathbb{T}$-gain graph with $n$ vertices and the maximum vertex degree $\Delta(\Phi)\geq 3$. Then
	\begin{equation*}
		r(\Phi)\geq \frac{n}{\Delta-1}
	\end{equation*}
	equality holds if and only if $\Phi\notin \{(K_{\frac{n}{2},\frac{n}{2}},1),(K_{\frac{n+1}{2},\frac{n-1}{2}},1) \}$.
\end{theorem}
\begin{proof}
	Let $\Phi$ be any connected $\mathbb{T}$-gain graph. Then, by Lemma \ref{lm3.2} and Theorem \ref{th2.1},
	\begin{equation}\label{eq2}
		\eta(\Phi)\leq M(\Phi)\leq Z(\Gamma(\Phi))\leq \frac{(\Delta-2)n}{\Delta-1}.
	\end{equation}
	Now the right most inequality in  (\ref{eq2}) holds if and only if $\Gamma(\Phi) \notin \{ G_1, G_2, K_n, K_{\frac{n}{2},\frac{n}{2}}, K_{\frac{n+1}{2},\frac{n-1}{2}}\}$. Therefore, if $\Gamma(\Phi) \notin \{ G_1, G_2, K_n, K_{\frac{n}{2},\frac{n}{2}}, K_{\frac{n+1}{2},\frac{n-1}{2}}\}$, then $ r(\Phi)\geq \frac{n}{\Delta-1} $ holds. From (\ref{eq2}),   $\eta(\Phi) \leq\frac{(\Delta-2)n}{\Delta-1}$ is possible, even if $Z(\Gamma(\Phi)) > \frac{(\Delta-2)n}{\Delta-1}$. So, for some of the  $ \mathbb{T} $-gain graphs whose underlying graphs in $\{ G_1, G_2, K_n, K_{\frac{n}{2},\frac{n}{2}}, K_{\frac{n+1}{2},\frac{n-1}{2}}\}$ may have rank greater or equal to $\frac{n}{\Delta-1}$. Let $L=\frac{n}{\Delta-1}.$
	
	\noindent {\bf Case 1:}  Let $\Phi$ be a gain graph with $\Gamma(\Phi)=G_1.$ Then $L=2.5$. Suppose $ r(\Phi) =2$. Then by Lemma \ref{lm2.4}, all induced subgraphs of $ \Phi $ have rank at most $ 2 $. Since the rank of any non empty graph is at least $ 2 $. So all non empty induced subgraphs of $ \Phi $ have rank $ 2 $. Consider two cyclic subgraphs of length $ 4 $, namely $\Phi_1=(C_1, \varphi)$ and $ \Phi_2=(C_2, \varphi)$, where $ C_1 \equiv v_1-v_2-v_3-v_4-v_1 $ and $ C_2\equiv v_1-v_2-v_3-v_5-v_1 $. Now $ r(\Phi_1)=r(\Phi_2)=2 $. Therefore, by Theorem 2.3, $ \Phi_1 $ and $ \Phi_2 $ are balanced. Hence both the $ 3 $-cycles in $ \Phi $ are balanced. Now the rank of any balanced $ 3 $-cycles is $ 3 $, a contradiction.   Thus $ r(\Phi)\geq \frac{n}{\Delta-1} $ holds for any $ \Phi $.
	%{\bf If $ r(\Phi)=2$ , $ \Phi $ can not be multipartite always. That is why I give separate proof of case 1.}
	
	\noindent{\bf Case 2:}  Let $\Phi$ be a gain graph with $\Gamma(\Phi)=G_2.$
	Then $ L=\frac{7}{3}=2.333 $. Since $ \Phi $ can not have rank $ 2 $, so $ r(\Phi)\geq \frac{n}{\Delta-1} $ holds for any $ \Phi $.
	
	\noindent{\bf Case 3:} Let $\Phi$ be a gain graph with $\Gamma(\Phi)=K_n.$
	Then $L=1+\frac{2}{n-2}$. Therefore, $ r(\Phi)\geq 2\geq \frac{n}{\Delta-1} $ holds for any $ \Phi $.
	
	\noindent{\bf Case 4:} Let $\Phi$ be a gain graph with $\Gamma(\Phi)=K_{\frac{n}{2},\frac{n}{2}}.$ Then $L>2$. Therefore $r(\Phi)\geq 3$. Using Case 3 of Theorem \ref{main_result}, the $\mathbb{T}$-gain graph $\Phi$ is of rank $2$ if and only if $\Phi$ is balanced.  Therefore, inequality holds for any $\mathbb{T}$-gain graph $\Phi=(K_{\frac{n}{2},\frac{n}{2}},\varphi)$ except $(K_{\frac{n}{2},\frac{n}{2}},1)$.
	
	\noindent{\bf Case 5:} Let $\Phi$ be a gain graph with $\Gamma(\Phi)=K_{\frac{n+1}{2},\frac{n-1}{2}}$. Then $L>2$. Therefore, similar to case 4, $ r(\Phi)\geq \frac{n}{\Delta-1} $ holds for any $ \Phi $ except $ (K_{\frac{n+1}{2},\frac{n-1}{2}}, 1) $.

\end{proof}

\begin{remark}
	It is easy to see that $\frac{n+1}{\Delta} \leq \frac{n}{\Delta-1}$ holds.  The bound for $r(\Phi)$ in Theorem \ref{th3.2} is better than that of  Theorem \ref{Th1.3}.
\end{remark}

%========================================

%\begin{theorem}
%Let $ \Phi=(G, \varphi)$ be any connected $\mathbb{T}$-gain graph of $n$ vertices with rank $r(\Phi)$ and largest vertex degree $\Delta(\Phi)\geq 2$. Then
%$$ r(\Phi)\geq\frac{n-2}{\Delta-1}$$ and equality occur if and only if $\Phi=(C_n,1)$, $n$ is divisible by 4 or $\Phi=(K_{\frac{n}{2}, \frac{n}{2}},1)$
%\end{theorem}
%
%\begin{theorem}
%Let $\Phi=(G,\varphi)$ be any connected $\mathbb{T}$-gain graph of order $n$ with the largest vertex degree $\Delta\geq3$ and rank $r(\Phi)$. Then $$r(\Phi)\geq \frac{n}{\Delta-1}$$ if and only if $\Phi\notin\{ aaa\}$
%\end{theorem}

%===========================================
\section*{Acknowledgments}
Aniruddha Samanta thanks University Grants Commission(UGC)  for the financial support in the form of the Senior Research Fellowship (Ref.No:  19/06/2016(i)EU-V; Roll No. 423206). M. Rajesh Kannan would like to thank the SERB, Department of Science and Technology, India, for financial support through the projects MATRICS (MTR/2018/000986) and Early Career Research Award (ECR/2017/000643).

\bibliographystyle{amsplain}
\bibliography{raj-ani-ref1}

\end{document}